\documentclass[11pt]{amsart}

\newcommand{\area}{{\rm Area}}
\newcommand{\vol}{{\rm Vol}}

\usepackage{amsthm,amsmath,amsfonts, amssymb}
\usepackage[all]{xy}
\usepackage{psfrag}
\usepackage[dvips]{graphicx}

\usepackage[latin1]{inputenc}

\makeatletter
\def\th@alexnormal{%
\let\thm@indent\noindent 
\thm@headfont{\bfseries}
\normalfont
}
\makeatother

\makeatletter
\def\th@alexit{%
\let\thm@indent\noindent 
\thm@headfont{\bfseries}
\normalfont
\fontshape{it}
\selectfont
}
\makeatother

\theoremstyle{alexit}

\newtheorem{theorem}[equation]{Theorem}
\newtheorem{proposition}[equation]{Proposition}
\newtheorem{lemma}[equation]{Lemma}

\newtheorem{corollary}[equation]{Corollary}

\newtheorem{conj}[equation]{Conjecture}
\theoremstyle{remark}
\newtheorem{remark}[equation]{Remark}
\theoremstyle{definition}
\newtheorem{definition}[equation]{Definition}

\numberwithin{equation}{subsection}

\begin{document}
\author{Alexandre Girouard}
\address{D\'epartement de Math\'ematiques et
Statistique, Universit\'e de Montr\'eal, C. P. 6128,
Succ. Centre-ville, Montr\'eal, Canada H3C 3J7}
\email{girouard@dms.umontreal.ca}
\author{Nikolai Nadirashvili}
\address{Laboratoire d'Analyse, Topologie, Probabilit\'{e}s UMR 6632,
Centre de Math\'{e}matiques et Informatique, Universit\'{e} de
Provence, 39 rue F. Joliot-Curie, 13453 Marseille Cedex 13, France;}
\email{nicolas@cmi.univ-mrs.fr}
\author{Iosif Polterovich}
\address{D\'epartement de Math\'ematiques et
Statistique, Universit\'e de Montr\'eal, C. P. 6128,
Succ. Centre-ville, Montr\'eal, Canada H3C 3J7}
\email{iossif@dms.umontreal.ca}

\title[Maximizaton of the second positive Neumann
eigenvalue]{Maximization of the second positive Neumann eigenvalue for
  planar domains}
\date{\today}
\begin{abstract} We prove that the second positive Neumann eigenvalue
of a bounded simply-connected planar domain of a given area does
not exceed the first positive Neumann eigenvalue on a disk of a
twice smaller area. This estimate is sharp and attained by a
sequence of domains degenerating to a union of two identical
disks. In particular, this result implies the Polya conjecture for the
second Neumann eigenvalue. The proof is based on a combination of
analytic and topological arguments. As a by-product of our method
we obtain an upper bound on the second eigenvalue for conformally
round metrics on odd-dimensional spheres.
\end{abstract}
\maketitle

\section{Introduction and main results}
\subsection{Neumann eigenvalues of planar domains} \label{intro1}
Let $\Omega$ be a bounded planar domain. The domain $\Omega$ is said
to be {\it regular} if the spectrum of the Neumann boundary value
problem on $\Omega$ is discrete. This is true, for instance, if
$\Omega$ satisfies the cone condition, that is there are no outward
pointing cusps (see \cite{NS} for more refined conditions and a
detailed discussion).

Let $0=\mu_0<\mu_1(\Omega)\leq\mu_2(\Omega)\leq\cdots\nearrow\infty$
be the Neumann eigenvalues of a regular domain $\Omega$. According
to a classical result of Szegö (\cite{Szego1}, see also \cite[p.
137]{SY}, \cite[section 7.1]{Henrot}), for any regular
simply-connected domain $\Omega$
\begin{equation}
\label{szego} \mu_1(\Omega) \, \area(\Omega)  \le \mu_1({\mathbb D})
\pi \approx 3.39\, \pi, \end{equation}
 where ${\mathbb D}$ is the unit disk, and
$\mu_1(\mathbb{D})$ is the square of the first zero of the
derivative $J_1'(x)$ of the first Bessel function of the first type.
The proof of Szeg\"o's theorem relies on the Riemann mapping theorem
and hence works only if $\Omega$ is simply-connected. However,
inequality \eqref{szego} holds without this assumption, as was later
shown by Weinberger \cite{Wein}.

The P\'olya conjecture for Neumann eigenvalues \cite{Polya1} (see
also \cite[p. 139]{SY}) states that for any regular bounded domain
$\Omega$
\begin{equation}
\label{polya} \mu_k(\Omega)\, \area(\Omega) \le 4 k\, \pi\,
\end{equation}
for all $k\ge 1$. This inequality is true for all domains that tile
the plane, e.g., for any triangle and any quadrilateral
\cite{Polya2}. It follows from the two-term asymptotics for the
eigenvalue counting function (\cite{Ivrii}, \cite{Melrose}) that for
any domain there exists a number $K$ such that \eqref{polya} holds
for all $k>K$.

Inequality \eqref{szego} implies that \eqref{polya} is true for
$\mu_1$. The best one could show for $k\ge 2$ was  $\mu_k \le 8\pi
k$ (\cite{Kroger}). In the present paper we consider the case $k=2$.
Our main result is
\begin{theorem}\label{maintheorem}
\label{main} Let $\Omega$ be a regular simply-connected planar
domain. Then
\begin{equation}
\label{main:bound} \mu_2(\Omega)\, \area(\Omega)\leq 2\,
\mu_1(\mathbb{D})\,\pi \approx 6.78 \, \pi,
\end{equation}
with the equality attained in the limit by a family of domains
degenerating to a disjoint union of two identical disks.
\end{theorem}
The second part of the theorem immediately follows from \eqref{main:bound}. Indeed, if $\Omega$ is a disjoint union of two identical disks
then \eqref{main:bound} is an equality. Joining the two disks by a passage of width $\epsilon$ we can construct a family of
simply-connected domains such that the
left-hand side in \eqref{main:bound} tends to $2\mu_1(\mathbb{D})\pi$ as $\epsilon \to 0$.

Theorem \ref{main} gives a positive answer to a question of
Parnovski \cite{Par}, motivated by an analogous result proved in
\cite{Nad} for the second eigenvalue on
a sphere. 
Note that \eqref{main:bound} immediately implies \eqref{polya}  for
$k=2$ for any regular simply-connected planar domain.
\begin{remark} It would be interesting to check the  bound \eqref{main:bound}
for non-simply connected domains. We believe it remains true in this
case as well.
\end{remark}

\begin{remark} All estimates discussed in this section have analogues in the Dirichlet case.
For example,  \eqref{szego} is the Neumann counterpart of the
celebrated Faber-Krahn inequality (\cite{Faber, Krahn1}, see also
\cite[section 3.2]{Henrot}), which states that among all bounded
planar domains of a given area, the first Dirichlet eigenvalue is
minimal on a disk. Similarly, Theorem \ref{main} can be viewed as an
analogue of the result due to Krahn and Szeg\"o (\cite{Krahn2},
\cite[Theorem 4.1.1]{Henrot}), who proved that among bounded planar
domains of a given area, the second Dirichlet eigenvalue is
minimized by the union of two identical disks.
\end{remark}
\subsection{Eigenvalue estimates on spheres} \label{intro2}Let $({\mathbb S}^n, g)$
be a sphere of dimension $n\ge 2$ with a Riemannian metric $g$. Let
$$0 < \lambda_1({\mathbb S}^n, g) \le \lambda_2({\mathbb S}^n, g)\leq\cdots\nearrow\infty$$
be the eigenvalues of the Laplacian on $({\mathbb S}^n, g)$. Hersch
\cite{Hersch} adapted the approach of Szeg\"o to prove that
$\lambda_1({\mathbb S}^2, g)\, \area({\mathbb S}^2, g) \le 8\pi$ for
any Riemannian metric $g$, with the equality attained on a sphere
with the standard round metric $g_0$. In order to obtain a similar
estimate in higher dimensions, one needs to restrict the Riemannian
metrics to a fixed conformal class \cite{EI}. Indeed, in dimension
$\geq 3$, if one only restricts the volume, $\lambda_1$ is
unbounded~\cite{CD}. In particular, it was shown in \cite{EI} (see
also \cite{MaWu}) that for any metric $g$ in the class $[g_0]$ of
conformally round metrics,
\begin{equation}
\label{hersch:hd} \lambda_1({\mathbb S}^n, g)\, \vol({\mathbb S}^n,
g)^\frac{2}{n} \le n\, \omega_n^{2/n},
\end{equation}
where
$$
\omega_n=\frac{2\,\pi^{\frac{n+1}{2}}}{\Gamma\left(\frac{n+1}{2}\right)}
$$
is the volume of the unit round $n$-dimensional sphere. This
result can be viewed as a generalization of Hersch's inequality,
since all metrics on ${\mathbb S}^2$ are conformally equivalent to the
round metric $g_0.$

A similar problem for higher eigenvalues is much more complicated.
It was proved in \cite[Corollary 1]{ElCol} that
\begin{equation}
\label{conf} \lambda_k^c({\mathbb S}^n, [g_0]):=\sup_{g \in [g_0]}
\lambda_k({\mathbb S}^n, g)\, \vol({\mathbb S}^n, g)^\frac{2}{n} \ge
n\, (k\, \omega_n)^{2/n},
\end{equation}
The number $\lambda_k^c({\mathbb S}^n, [g_0])$ is called the  $k$-th
conformal eigenvalue of $({\mathbb S}^n, [g_0])$. It was shown in
\cite{Nad} that for $k=2$ and $n=2$ the inequality in \eqref{conf}
is an equality, and the supremum is attained by a sequence of
surfaces tending to a union of two identical round spheres. We
conjecture that the same is true in all dimensions:
\begin{conj}
\label{conj} The second conformal eigenvalue of $({\mathbb S}^n,
[g_0])$ equals
\begin{equation}
\label{Polya:sphere} \lambda_2^c({\mathbb S}^n, [g_0]) = n\, (2\,
\omega_n)^{2/n}
\end{equation}
for all $n\ge 2$.
\end{conj}
As a by-product of the method developed for the proof of Theorem
\ref{main}, we prove un upper bound for $\lambda_2^c({\mathbb S}^n,
[g_0])$ when the dimension $n$ is {\it odd} (this condition is
explained in Remark \ref{whyodd}). Our result is in good agreement
with Conjecture \ref{conj}.
\begin{theorem}
\label{sphere} Let $n \in \mathbb{N}$ be odd and let $({\mathbb
S}^n, g)$ be a $n$-dimensional sphere with a conformally round
metric $g \in [g_0]$. Then
\begin{equation}
\label{sphere:bound} \lambda_2({\mathbb S}^n, g)\, \vol({\mathbb
S}^n, g)^\frac{2}{n} <
(n+1)\left(\frac{4\pi^{\frac{n+1}{2}}\Gamma(n)}{\Gamma(\frac{n}{2})\Gamma(n+\frac{1}{2})}\right)^{2/n}
\end{equation}
\end{theorem}
\begin{remark}
Note that the Dirichlet energy is not conformally invariant in
dimensions $n \ge 3$ and therefore to prove Theorem \ref{sphere} we
have to work with the modified Rayleigh quotient (cf. \cite{FN}).
This is in fact the reason why we do not get a sharp bound (see
Remark \ref{whynotsharp}). At the same time, the estimate
\eqref{sphere:bound} is just a little bit weaker than the
conjectured bound \eqref{Polya:sphere}: one can check numerically
that the ratio of the constants at the right-hand sides of
\eqref{sphere:bound} and \eqref{Polya:sphere} is contained in the
interval $(1,1.04)$ for all $n$. Moreover, the difference between
the two constants tends to $0$ as the dimension $n \to \infty$, and
hence \eqref{sphere:bound} is ``asymptotically sharp'' as follows
from \eqref{conf}.
\end{remark}
\begin{remark}
It was conjectured in \cite{Nad} that if $n=2$ then \eqref{conf} is
an equality for all $k \ge 1$, with the maximizer given by the union
of $k$ identical round spheres. One could view it as an analogue of
the P\'olya conjecture \eqref{polya} for the sphere. Note that a
similar ``naive'' guess about the maximizer of the $k$-th Neumann
eigenvalue of a planar domain is false: a union of $k$ equal disks
can not maximize $\mu_k$ for all $k \ge 1$, because, as one could
easily check, this would contradict Weyl's law. For the same reason,
\eqref{conf} can not be an equality for all $k\ge 1$ in dimensions
$n\ge 5$.
\end{remark}

\subsection{Plan of the paper} The paper is organized as follows. In
sections 2.1--2.5 we develop the ``folding and rearrangement''
technique based on the ideas of \cite{Nad} and apply it to planar
domains. The topological argument used in  the proof of Theorem
\ref{maintheorem} is presented in section 2.6. In section 2.7 we
complete the proof of the main theorem using some facts about the
subharmonic functions. In sections 3.1 and 3.2 we prove the
auxiliary lemmas used in the proof of Theorem \ref{maintheorem}. In
section 4.1 we present a somewhat stronger version of the classical
Hersch's lemma (\cite{Hersch}). 
In sections 4.2
and 4.3 we adapt the approach developed in sections 2.1-2.7 for the
case of the sphere. In section 4.4 we use the modified Rayleigh
quotient to complete the proof of Theorem \ref{sphere}.

\subsection*{Acknowledgments} We are very grateful to L.~Parnovski
for a stimulating question that has lead us to Theorem \ref{main},
and to M.~Levitin for many useful discussions on this project. We
would also like to thank B.~Colbois and L.~Polterovich for helpful
remarks.

\section{Proof of Theorem \ref{maintheorem}}
\subsection{Standard eigenfunctions for $\mu_1$ on the disk} Let
$$\mathbb{D}=\left\{z\in\mathbb{C}\, \bigl|\bigr.\, |z|<1\right\}$$ be the open unit
disk.  Let $J_1$ be the first Bessel function of the first kind, and
let $\zeta \approx 1.84$ be the smallest positive zero of the
derivative $J_1'$.
Set
$$f(r)=J_1(\zeta r).$$ Given $R\geq 0$ and
$s=(R\cos\alpha, R\sin\alpha)\in\mathbb{R}^2,$ define
$X_s:\mathbb{D}\rightarrow\mathbb{R}$ by
\begin{equation}
\label{Xs}
 X_s(z)=f(|z|)\frac{z\cdot
s}{|z|}=Rf(r)\cos(\theta-\alpha),
\end{equation}
where
$r=|z|$, $\theta=\arg z$, and $z\cdot s$ denotes the scalar product
in $\mathbb{R}^2$. The functions $X_s$ are the Neumann
eigenfunctions corresponding to the double eigenvalue
$$\mu_1(\mathbb{D})=\mu_2(\mathbb{D})=\zeta^2\approx 3.39.$$ The
functions $X_{e_1}$ and $X_{e_2}$ form a basis for this space of
eigenfunctions (where the vectors $\{e_1, e_2\}$ form the standard
basis of $\mathbb{R}^2$).

\subsection{Renormalization of measure}
We say that a conformal transformation $T$ of the disk {\it
renormalizes} a measure $d\nu$ if for each $s\in\mathbb{R}^2$,
  \begin{gather}\label{DefTrenormalizes}
    \int_{\mathbb{D}}X_s\circ T\,d\nu=0.
  \end{gather}

Finite signed measures on $\mathbb{D}$ can be seen as elements of the
dual of the space $C(\overline{\mathbb{D}})$ of continuous
functions. As such, the norm of a measure $d\nu$ is
\begin{equation}
\label{norm}
  \|d\nu\|=\sup_{f\in C(\overline{\mathbb{D}}), |f|\le 1} \left|\int_{\mathbb{D}}f\,d\nu\, \right|
\end{equation}

The following result is an analogue of  Hersch's lemma (see
~\cite{Hersch},~\cite{SY}).
\begin{lemma}\label{renormalization}
For any finite measure $d\nu$ on $\mathbb{D}$ there exists a
    point $\xi \in\mathbb{D}$ such that $d\nu$ is renormalized by the automorphism
    $d_\xi:\mathbb{D}\rightarrow\mathbb{D}$ defined by
    $$d_\xi(z)=\frac{z+\xi}{\overline{\xi}z+1}.$$
\end{lemma}
\begin{proof}
  \noindent
Set $M=\int_\mathbb{D}d\nu$ and define the continuous map
  $C:\mathbb{D}\rightarrow\mathbb{D}$ by
  \begin{align*}
      C(\xi&)=\frac{1}{M\,f(1)}\int_{\mathbb{D}}\left(X_{e_1},X_{e_2}\right)\,(d_{\xi})_*d\nu
      =\frac{1}{M\,f(1)}\int_{\mathbb{D}}\left(X_{e_1}\circ d_\xi,X_{e_2}\circ d_\xi\right)d\nu
  \end{align*}
  Let $e^{i\theta}\in S^1=\partial\mathbb{D}$.
  For any $z\in\mathbb{D}$,
  $$\lim_{\xi\rightarrow e^{i\theta}} d_\xi(z)=e^{i\theta}.$$
  This means that the map
  $C$ can be continuously extended to the closure $\overline{\mathbb{D}}$ by
  $C=\mbox{id}$ on $\partial\mathbb{D}$.
  By the same topological argument as in Hersch's lemma (and as in the
  proof of the Brouwer fixed point theorem), a continuous map
  $C:\overline{\mathbb{D}}\rightarrow\overline{\mathbb{D}}$ such
  that $C(\xi)=\xi$ for $\xi\in\partial\mathbb{D}$ must be
  onto. Hence, there exists some $\xi\in\mathbb{D}$ such that
  $C(\xi)=0\in\mathbb{D}$. \end{proof}

\begin{lemma}\label{uniqueness}
  For any finite measure $d\nu$ the renormalizing point $\xi$ is
  unique.
\end{lemma}
\begin{proof}
  \noindent
  First, let us show that if the measure $d\nu$ is already
  renormalized then $\xi=0$. Suppose that $\mathbb{D}\ni \eta \ne 0$
  renormalizes $d\nu$. Without loss of generality assume that
  $\eta$ is real and positive (if not, apply a rotation). Setting $s=1$, by
  Lemma \ref{aux1} we get that  $X_s(d_\eta(z))>X_s(z)$ for all $z\in
  \mathbb{D}$ and hence
  $$\int_{\mathbb{D}}X_s\circ d_\eta\,d\nu>\int_{\mathbb{D}}X_s\,d\nu=0,$$
  which contradicts the hypothesis that $\eta$ renormalizes $d\nu$.

  Now let $d\nu$ be an arbitrary finite measure which is renormalized
  by $\xi\in \mathbb{D}$. Assume $\eta\in \mathbb{D}$ also
  renormalizes $d\nu$. Let us show that $\eta=\xi$. Taking into
  account that $d_{-\xi}\circ d_\xi=d_0=\mbox{id}$, we can write
  $$(d_{\eta})_*d\nu=\left(d_{\eta}\circ d_{-\xi}\right)_*\left(d_{\xi}\right)_*d\nu.$$
  A straightforward computation shows that $$d_{\eta}\circ
  d_{-\xi}=\frac{1-\eta \bar \xi}{1-\bar \eta \xi} d_\alpha,$$ where
  $\alpha=d_{-\xi}(\eta)$ and
  $\left|\frac{1-\eta \bar \xi}{1-\bar \eta\xi}\right|=1$.
  This implies that $d\alpha$ renormalizes
  $\left(d_{\xi}\right)_*d\nu$ which is already renormalized.  Hence,
  as we have shown above, $\alpha=d_{-\xi}(\eta)=0$, and therefore
  $\xi=\eta$.
\end{proof}
Given a finite measure, we write
$\Gamma(d\nu)\in\mathbb{D}$
for its unique renormalizing point $\xi\in\mathbb{D}$.
\begin{corollary}
  The renormalizing point $\Gamma(d\nu)\in\mathbb{D}$ depends continuously on the
  measure $d\nu$.
  \label{cont}
\end{corollary}
\begin{proof}
  \noindent
  Let $(d\nu_n)$ be a sequence of measures converging to the measure
  $d\nu$ in the norm \eqref{norm}. Without loss of generality suppose that $d\nu$ is
  renormalized. Let $\xi_n\in\mathbb{D}\subset\overline{\mathbb{D}}$
  be the unique element such that $d_{\xi_n}$ renormalizes $d\nu_n$.
  Let $(\xi_{n_k})$ be a convergent subsequence, say to
  $\xi\in\overline{\mathbb{D}}$.
  Now, by definition of $\xi_n$ there holds
  \begin{align*}
    0=\lim_{k\rightarrow\infty}|\int_{\mathbb{D}}X_{s}\,(d_{\xi_{n_k}})_*d\nu_{n_k}|=
    |\int_{\mathbb{D}}X_{s}\,(d_{\xi})_*d\nu|,
  \end{align*}
  and hence $d_\xi$ renormalizes $d\nu$. Since we assumed that $d\nu$
  is normalized, by uniqueness we get $\xi=0$. Therefore, $0$ is the
  unique accumulation point of the set $\xi_n\in \mathbb{D}$ and hence
  by compactness we get $\xi_n \to 0$. This completes the proof of the
  lemma.
\end{proof}

Corollary \ref{cont} will be used in the proof of Lemma
\ref{flipfloplemma}, see section \ref{section:flip}.
\subsection{Variational characterization of $\mu_2$}
\label{section:var} It follows from the Riemann mapping theorem and
Lemma~\ref{renormalization} that for any simply-connected domain
$\Omega$ there exists a conformal equivalence
$\phi:\mathbb{D}\rightarrow\Omega$, such that the pullback measure
$$d\mu(z)=\phi^*(dz)=|\phi'(z)|^2\,dz$$
satisfies for any $s\in
S^1$
\begin{gather}\label{conventionI}
  \int_\mathbb{D}X_s(z)\, d\mu(z)=0.
\end{gather}
Using a rotation if necessary, we may also assume that
\begin{gather}\label{conventionII}
    \int_{\mathbb{D}}X_{e_1}^2(z)\,d\mu(z)\geq\int_{\mathbb{D}}X_{s}^2(z)\,d\mu(z).
\end{gather}
for any $s\in S^1$. The proof of Theorem~\ref{maintheorem} is based
on the following variational characterization of $\mu_2(\Omega)$:
\begin{gather}\label{varicharact}
  \mu_2(\Omega)=\inf_{E}
  \sup_{0\neq u\in E}\frac{\int_{\mathbb{D}}|\nabla
    u|^2\,dz}{\int_{\mathbb{D}}u^2\,d\mu}
\end{gather}
where $E$ varies among all two-dimensional subspaces of
the Sobolev space $H^1({\mathbb{D}})$ that are orthogonal to
constants, that is for each $f\in E$,
$\int_{\mathbb{D}}f\,d\mu=0$. Note that the Dirichlet energy is
conformally invariant
in two dimensions, and hence the numerator in \eqref{varicharact} can
be written using the standard Euclidean gradient and the Lebesgue
measure.

\subsection{Folding of hyperbolic caps}
It is well-known that the group of automorphisms of the disk
coincides with the isometry group of the Poincar\'e disk model of
the hyperbolic plane \cite[section 7.4]{Be}. Therefore, for
any $\xi\in\mathbb{D}$, the automorphism
$$d_\xi(z)=\frac{z+\xi}{\overline{\xi}z+1}$$
is an isometry. Note that we have  $d_0=\mbox{id}$ and
$d_\xi(0)=\xi$ for any $\xi$.

Let $\gamma$ be a geodesic in the Poincar\'e disk model, that is a
diameter or the intersection of the disk with a circle which is
orthogonal to $\partial \mathbb{D}$. Each connected
component of $\mathbb{D}\setminus\gamma$ is called a
\emph{hyperbolic cap} on $\mathbb{D}$. The space of hyperbolic caps
is parametrized as follows. Given $(r,p)\in(-1,1)\times S^1$ let
$$a_{r,p}=d_{rp}(a_{0,p}),$$ where
$$a_{0,p}=\{x\in\mathbb{D}\ :\ x\cdot p>0 \}$$
 is the half-disk such that $p$ is the center of its boundary half-circle. The
limit $r\rightarrow 1$ corresponds to a cap degenerating to a point
on the boundary $\partial\mathbb{D}$ (that is, $a\rightarrow p$),
while the limit $r\rightarrow -1$ corresponds to degeneration to the
full disk $\mathbb{D}$ (that is, $a\rightarrow\mathbb{D}$).
\begin{figure}[h]
  \centering
  \psfrag{r}[][][1]{$r$}
  \psfrag{p}[][][1]{$p$}
  \psfrag{a}[][][1]{$a_{0,p}$}
  \psfrag{b}[][][1]{$a_{r,p}$}
  \psfrag{d}[][][1]{$\stackrel{d_{rp}}{\longrightarrow}$}
  \includegraphics[width=10cm]{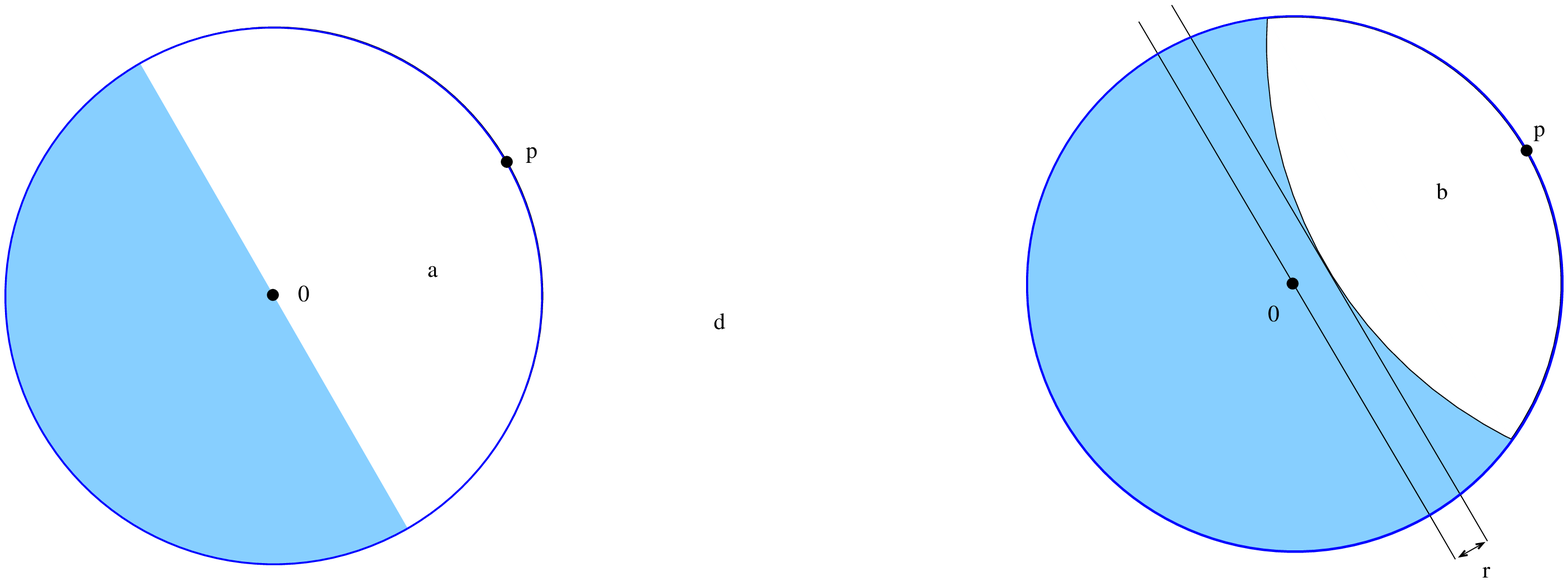}
\end{figure}
Given $p\in \mathbb{D}$, we define the automorphism  $R_p(z)=-p^2
\bar z$. It is the reflection with respect to the line going through
$0$ and orthogonal to the segment joining $0$ and $p$. For each cap
$a_{r,p}$, let us define a conformal automorphism
\begin{gather*}
  \tau_{a}=d_{rp}\circ R_p\circ d_{-rp}.
\end{gather*}
One can check that this is the reflection with respect to the
hyperbolic geodesic $\partial a_{r,p}$. In particular,
$\tau_a(a)=\mathbb{D}\setminus\overline{a}$ and $\tau_a$ is the
identity on $\partial a$.
\subsection{Folding and rearrangement of measure}
\label{section:folding}
  Given a measure $d\mu$ on $\mathbb{D}$ and a hyperbolic cap $a\subset\mathbb{D}$, the \emph{folded
    measure} $d\mu_a$ is defined by
  $$d\mu_a=
  \begin{cases}
    d\mu+\tau_a^*d\mu & \mbox{ on } a,\\
    0 & \mbox{ on } \mathbb{D}\setminus\overline{a}.
  \end{cases}$$
Clearly, the measure $d\mu_a$ depends continuously in the norm
\eqref{norm} on the cap $a\subset\mathbb{D}$.
For each cap $a \in \mathbb{D}$ let us construct the following
conformal equivalence $\psi_a: \mathbb{D} \to a$.
\begin{figure}[h]
  \centering
  \psfrag{f}[][][1]{$\stackrel{T_a}{\longrightarrow}$}
  \psfrag{g}[][][1]{$\stackrel{\phi_b}{\longrightarrow}$}
  \psfrag{h}[][][1]{$\stackrel{T'_a}{\longrightarrow}$}
  \psfrag{w}[][][1]{$\psi_a$}
  \psfrag{a}[][][1]{$a$}
  \psfrag{b}[][][1]{$b$}
  \psfrag{D}[][][1]{$\mathbb{D}$}
  \includegraphics[width=10cm]{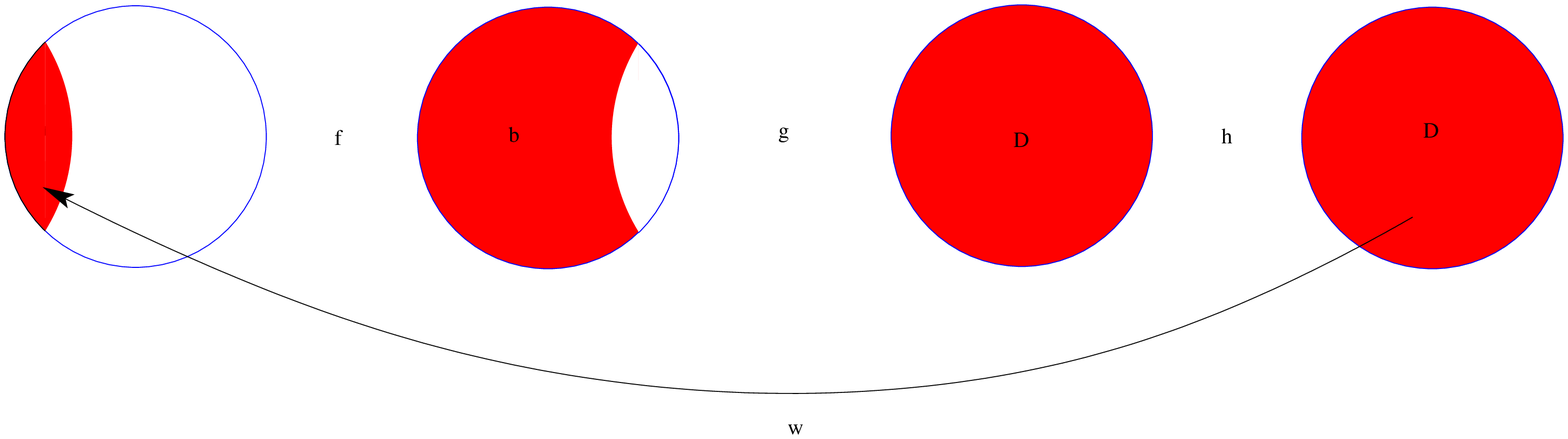}
\end{figure}
First, observe that it follows from the proof of the Riemann mapping
theorem~\cite[p.342]{Taylor} that there exists  a family
$\phi_a:a\rightarrow\mathbb{D}$ of conformal equivalences depending
continuously on the cap $a$ such that
$\displaystyle\lim_{a\rightarrow\mathbb{D}}\phi_a=\mbox{id}$ pointwise.
Let $\xi(a)=\Gamma(d\mu_a)$ be the normalizing point for the measure
$d\mu_a$ and set $T_a=d_{\xi(a)}$. The measure $(T_a)_*d\mu_a$ is
supported in the cap $b=T_a(a)$.
Pushing this measure to the full
disk using $\phi_b:b \rightarrow\mathbb{D}$ leads to the measure
$$(\phi_b\circ T_a)_*d\mu_a.$$
Let $\eta(a)=\Gamma((\phi_b \circ T)_*d\mu_a)$ and set
$$T'_a:=d_{\eta(a)}:\mathbb{D}\rightarrow\mathbb{D}$$
The conformal equivalence $\psi_a:\mathbb{D}\rightarrow a$ is
defined by
$$\psi_a=\left(T'_a\circ\phi_b\circ T_a\right)^{-1}.$$

The pull-back by $\psi_a$ of the folded measure is
\begin{equation}
\label{rearranged}
 d\nu_a=\psi_a^*d\mu_a
\end{equation}
It is clear from the above construction that
$d\nu_a$ is a normalized measure on the whole disk. We call $d\nu_a$
the {\it rearranged measure}.
It also follows from the construction that the conformal
transformations $\psi_a: \mathbb{D}\rightarrow a$ depend continuously
on $a$ and
\begin{equation}
\label{fulldisk}
  \lim_{a\rightarrow\mathbb{D}}\psi_a =
  \mbox{id}:\mathbb{D}\rightarrow\mathbb{D}
  \end{equation}
in the sense of the pointwise convergence. We will make use of the
following important property of the rearranged measure.
\begin{lemma}\label{flipfloplemma} If a sequence of hyperbolic caps
$a \in \mathbb{D}$  degenerates to a point $p \in \partial
\mathbb{D}$, the limiting rearranged measure is a ``flip-flop'' of
the original measure $d\mu$:
  \begin{gather}
    \lim_{a\rightarrow p}d\nu_a=R_p^*d\mu. \tag{F}\label{fliflop}
  \end{gather}
\end{lemma}
We call~(\ref{fliflop})  the flip-flop property. The  proof of
Lemma~\ref{flipfloplemma} will be presented at the end of the paper.

\subsection{Maximizing directions}
Given a finite measure $d\nu$ on $\mathbb{D}$, consider
the function $V:\mathbb{R}^2\rightarrow\mathbb{R}$ defined by
$$V(s)=\int_{\mathbb{D}}X_s^2\, d\nu.$$
This function is a quadratic form since the
mapping $\mathbb{R}^2 \times \mathbb{R}^2 \to \mathbb{R}$ defined by
$$(s,t) \mapsto\int_{\mathbb{D}}X_s X_t\, d\nu$$
is symmetric and bilinear (the latter easily follows from
\eqref{Xs}). In particular, $V(s)=V(-s)$ for any $s$.

Let $\mathbb{R}P^1=S^1/\mathbb{Z}_2$ be the projective line. We
denote by $[s]\in \mathbb{R}P^1$ the element of the projective line
corresponding to the pair of points $\pm s \in S^1$.
We say that $[s] \in \mathbb{R}P^1$ is a {\it maximizing
direction} for the measure $d\nu$ if $V(s)\ge V(t)$ for any
$t \in S^1$.
The measure $d\nu$ is called {\it simple} if there is
a unique maximizing direction. Otherwise, the measure $d\nu$ is said
to be {\it multiple}.
We have the following
\begin{lemma}
\label{multcaplemma}
  A measure $d\nu$ is multiple if and only if $V(s)$ does not depend on
  $s \in S^1$.
\end{lemma}
\begin{proof}
  Since $V(s)$ is a symmetric quadratic form, it can be
  diagonalized. This means that there exists an orthonormal basis
  $(v_1,v_2)$ of $\mathbb{R}^2$,
  such that for any $s=\alpha v_1+\beta v_2\in \mathbb{D}$ we have
  $V(s)=M\alpha^2+m\beta^2.$ for some numbers $0< m\leq M.$
  It is clear now that the measure $d\nu$  is multiple if and only if $M=m$,
  and therefore $V(s)$ takes the same value for all $s\in S^1$.
\end{proof}

Note that by \eqref{conventionII}, $[e_1]$ is a maximizing
direction for the measure $d\mu$.
\begin{proposition}
  \label{multexist} If the measure $d\mu$ is simple, then there exists
  cap $a\subset\mathbb{D}$ such that the rearranged measure $d\nu_a$ is
  multiple.
\end{proposition}
The proof of this proposition is based on a topological argument,
somewhat more subtle than the one used in the proof of Lemma
\ref{renormalization}. This is a proof by contradiction. We assume
the measure $d\mu$ as well as the measures $d\nu_a$ to be simple.
Given a cap $a\subset\mathbb{D}$,  let $[s(a)]\in\mathbb{R}P^1$ be
the unique maximizing direction for $d\nu_a$. Since $d\nu_a$ depends
continuously on $a$ and $X_s$ depends continuously on $s$, it
follows that the map $a\mapsto [s(a)]$ is continuous. Let us
understand the behavior of the maximizing directions as the cap $a$
degenerates to the full disk and to a point.
\begin{lemma}\label{degeneratecaps}
  Assume the measures $d\mu$ as well as each $d\nu_a$ to be simple.
  Then
  \begin{gather}
    \lim_{a\rightarrow\mathbb{D}} [s(a)]=[e_1]\label{eqndegenerate1}\\
    \lim_{a\rightarrow e^{i\theta}} [s(a)]=
    [e^{2i\theta}]\label{eqndegenerate2}.
  \end{gather}
\end{lemma}
\begin{proof}
  First, note that formula~\eqref{eqndegenerate1} immediately follows from
  \eqref{fulldisk} and \eqref{conventionII}.
  Let us prove \eqref{eqndegenerate2}. Set $p=e^{i\theta}$.
  Lemma~\ref{flipfloplemma} implies
  \begin{equation}
    \label{Xps}
    \lim_{a\rightarrow p}\int_{\mathbb{D}}X_s^2\,d\nu_a=
    \int_{\mathbb{D}}X_s^2\,R_p^*d\mu
    =\int_{\mathbb{D}}X_s^2\circ R_p\,d\mu
    =\int_{\mathbb{D}}X_{R_ps}^2\,d\mu.
  \end{equation}
  Since $[e_1]$ is the unique maximizing direction for $\mathbb{D}$,
  the right hand side of \eqref{Xps} is maximal for
  $R_ps=\pm e_1$. Applying $R_p$ on both sides we get
  $s=\pm e^{2i\theta}$ and hence  $[s]=[e^{2i\theta}]$.
\end{proof}
\begin{proof}[Proof of Proposition~\ref{multexist}]
  Suppose that for each cap $a\subset\mathbb{D}$ the measure $d\nu_a$
  is simple.
  Recall that the space of caps is identified with
  $(-1,1)\times S^1$.
  Define
  $h:(-1,1)\times S^1\rightarrow \mathbb{R}P^1$ by
  by $h(r,p)=[s(a_{r,p})].$
  It follows from Lemma~\ref{degeneratecaps}) that  $h$ extends to a continuous
  map on $[-1,1]\times S^1$ such that
  $$h(-1,e^{i\theta})=[e_1], h(1,e^{i\theta})=[e^{2i\theta}].$$
  This means that $h$ is a homotopy between a trivial loop and a
  non-contractible loop on $\mathbb{R}P^1$. This is a contradiction.
\end{proof}

\subsection{Test functions}
\label{subsection:testfunctions}
Assume that $d\mu$ is simple. By
Proposition~\ref{multexist} and Lemma~\ref{multcaplemma} there
exists a cap $a\subset\mathbb{D}$ such that
$$\int_{\mathbb{D}}X_s^2\,d\nu_a(z)$$
does not depend on the choice of $s\in S^1$.
Let $a^*=\mathbb{D}\setminus\overline{a}$.
\begin{definition}\label{defLift}
  Given a function $u:a\rightarrow\mathbb{R}$, the \emph{lift}
     of $u$, $\tilde{u}:\mathbb{D}\rightarrow\mathbb{R}$ is given by
  $$\tilde{u}(z)=
  \begin{cases}
    u(z) & \mbox{if } z\in a,\\
    u(\tau_az) & \mbox{if } z\in a^*.
  \end{cases}$$
\end{definition}
Given $u:a\rightarrow\mathbb{R}$ we have
\begin{gather*}
  \int_{a}u\,d\mu_a=\int_{a}u\,d\mu+\int_{a^*}u\circ\tau_a\,d\mu=\int_{\mathbb{D}}\tilde{u}\,d\mu,\\
\end{gather*}

For every $s\in\mathbb{R}^2$, set
$$u_a^s= X_{s} \circ \psi_a^{-1}:a\rightarrow\mathbb{R}.$$
We will use the two-dimensional space
$$E=
\left\{\tilde{u}_a^s \,\bigl|\bigr.
  s\in\mathbb{R}^2\right\}$$
of test functions in the variational
characterization~(\ref{varicharact}) of $\mu_2.$

\begin{proposition}\label{mainestimate}
  For each $s\in\mathbb{R}^2$
  \begin{equation}
  \label{fff}
  \frac{\int_{\mathbb{D}}|\nabla \tilde{u}_a^s|^2\,dz}{\int_{\mathbb{D}}(\tilde{u}_a^s)^2\,d\mu}
  \leq 2\mu_1(\mathbb{D}).
  \end{equation}
\end{proposition}
We split the proof of Proposition~\ref{mainestimate}  in two parts.
\begin{lemma}\label{dirichletestimate}
  For any hyperbolic cap $a\subset\mathbb{D}$,
  $$\int_{\mathbb{D}}|\nabla \tilde{u}_a^s|^2\,dz=\left(2\pi\int_{r=0}^1f^2(r)r\,dr\right)\mu_1(\mathbb{D}).$$
\end{lemma}
\begin{lemma}\label{l2estimate}
  \begin{gather}
    \int_{\mathbb{D}}(\tilde{u}_a^s)^2\,d\mu\geq\pi\left(\int_{r=0}^1f^2(r)r\,dr\right).
  \end{gather}
\end{lemma}
\begin{proof}[Proof of Lemma~~\ref{dirichletestimate}]
  It follows from the definition of the lift that
  \begin{align*}
    \int_{\mathbb{D}}|\nabla \tilde{u}_a^s|^2\, dz=
    \int_{a}|\nabla u_a^s|^2\,dz+\int_{a^*}|\nabla (u_a^s\circ\tau_a)|^2\,dz.
  \end{align*}
  By conformal invariance of the Dirichlet energy, the two terms on the right hand side are
  equal, so that
  \begin{align}
  \label{doubling}
    \int_{\mathbb{D}}|\nabla \tilde{u}_a^s|^2\, dz&=2\int_{a}|\nabla u_a^s|^2\,dz=
    2\int_{a}|\nabla (X_s\circ \psi_a^{-1})|^2\,dz\nonumber\\
    &=2\int_{\mathbb{D}}|\nabla X_s|^2\,dz\nonumber\hspace{.5cm}\longleftarrow\mbox{ (by conformal invariance)}\\
    &=2\mu_1(\mathbb{D})\int_{\mathbb{D}}X_s^2\,dz\hspace{.5cm}\longleftarrow(\mbox{since }X_s\mbox{ is the first eigenfunction on a disk})
  \end{align}
  It follows from \eqref{Xs} that given two orthogonal directions $s,t\in S^1$ we have
  $$\int_{\mathbb{D}}(X_{s}^2+X_{t}^2)\,dz=\int_{\mathbb{D}}f^2(|z|)\,dz.$$

 Therefore,  by symmetry we get
  \begin{align*}
      \int_{\mathbb{D}}X_s^2\,dz=\frac{1}{2}\int_{\mathbb{D}}f^2(|z|)\,dz
      =\pi\int_{r=0}^1f^2(r)r\,dr.
  \end{align*}
  This completes the proof of the lemma.
\end{proof}
To prove Lemma \ref{l2estimate} we use the following result.
\begin{lemma}
\label{subharm} The rearranged measure $d\nu_a$ on $\mathbb{D}$ can
be represented as $d\nu_a=\delta(z)dz$, where
$\delta:\mathbb{D}\rightarrow\mathbb{R}$ is a subharmonic function.
\end{lemma}
\begin{proof}
Indeed, $d\nu_a=\psi_a^*d\mu_a$, where the
  measure $d\mu_a$ on the cap $a$ is obtained as the sum of measures
  $d\mu$ and $\tau_a^* d\mu$.
  Both measures $d\mu$ and $\tau_a^* d\mu$
  correspond to flat Riemannian metrics on $a$, because $d\mu$ is the
  pullback of the Euclidean measure $dz$ on the domain $\Omega$ by the conformal map
   $\phi:\mathbb{D}\to \Omega$ (see section
  \ref{section:var}). Since the maps
  $\psi_a$ and $\tau_a$ are also conformal,
  one has $\psi_a^*d\mu=\alpha(z)dz$ and
  $\psi_a^*(\tau_a^*d\mu)=\beta(z)dz$ for some {\it subharmonic} functions
  $\alpha(z), \, \beta(z)$. Indeed, the
  metrics corresponding to these measures are flat
  (they are pullbacks by $\psi_a$ of flat metrics on $a$ that we mentioned above), and
  the well-known formula for the Gaussian curvature in isothermal coordinates
  yields $\Delta \log \alpha(z)=\Delta \log
  \beta(z)=0$ (cf. \cite[p. 663]{BR}).  Therefore, $\alpha(z)$ and $\beta(z)$ are
  subharmonic as exponentials of harmonic functions \cite[p. 45]{Levin}.
  Finally, $d\nu_a=\delta(z) dz$, where $\delta(z)=\alpha(z)+\beta(z)$ is subharmonic
  as a sum of subharmonic functions. This completes the proof of the
  lemma.
\end{proof}

\begin{proof}[Proof of Lemma~\ref{l2estimate}]
Set $$G(r)=\int_{B(0,r)}\delta(z)\,dz=\int_0^r \int_0^{2\pi}
\delta(\rho\, e^{i\phi})\rho\, d\rho\, d\phi.$$ By Lemma
\ref{subharm} the function $\delta$ is subharmonic. The
function
$$W(\rho)=\int_0^{2\pi} \delta(\rho \,e^{i\phi}) d\phi$$
is $2\pi$ times the average of $\delta$ over the circle of radius
$\rho$, hence it is monotone non-decreasing in $\rho$ (\cite[p.
46]{Levin}). Therefore, since $r\le 1$, we get as in
\cite[p.138]{SY} that
\begin{multline}
\label{inetchach} G(r)=\int_0^r W(\rho) \rho\, d\rho =\\ r^2
\int_0^1 W(r\, \rho) \rho \, d\rho\le r^2\int_0^1 W(\rho)\, \rho \,
d\rho = r^2 G(1)=\pi r^2.
\end{multline}
Now,  because $\tilde{u}_a^s$ is the lift of  $u_a^s=X_s\circ
\psi_a$, we have
  \begin{gather*}
    \int_{\mathbb{D}}(\tilde{u}_a^s)^2\,d\mu=\int_{a}(u_a^s)^2\,d\mu_a
    =\int_{\mathbb{D}}X_s^2\,d\nu_a.
  \end{gather*}
  Moreover since $V_a(s)$ doesn't depend on $s\in S^1$,
  \begin{align}
    V_a(s)=\int_{\mathbb{D}}X_s^2\,d\nu_a&=
    \frac{1}{2}\int_{\mathbb{D}}\left(X_{e_1}^2+X_{e_2}^2\right)\,d\nu_a\nonumber\\
    &=\frac{1}{2}\int_{\mathbb{D}}f^2(|z|)\,\delta(z)\,dz
    =\frac{1}{2}\int_{r=0}^1f^2(r)G'(r)\,dr\label{unestime}
  \end{align}
  Integrating by parts and taking into account that $G(r)\le \pi r^2$ due to \eqref{inetchach}, we get
    \begin{align}
    \int_{r=0}^1f^2(r)G'(r)\,dr
    &=f^2(1)G(1)-\int_{0}^1\frac{d}{dr}\bigl(f^2(r)\bigr)G(r)\,dr\,\ge \nonumber\\
    & f^2(1)G(1)-\pi\int_{0}^1\frac{d}{dr}\bigl(f^2(r)\bigr)r^2\,dr=2\pi\int_{0}^1f^2(r)r\,dr\nonumber\\
  \end{align}
  This completes the proof of Lemma \ref{l2estimate} and Proposition~\ref{mainestimate}.
\end{proof}

\smallskip

\begin{remark}
The proof of  Lemma \ref{l2estimate} is quite similar to the proof
of \eqref{szego},  see \cite[p. 348]{Szego1} and  \cite[p. 138]{SY}.
Our approach is somewhat more direct since it explicitly uses the
subharmonic properties of the measure.
\end{remark}

\smallskip

\begin{proof}[Proof of Theorem~\ref{maintheorem}]
Assume that $d\mu$ is simple.
Then \eqref{main:bound} immediately follows from Proposition
\ref{mainestimate} and the variational characterization
\eqref{varicharact} of $\mu_2$.

Suppose now that $d\mu$ is multiple. In fact, the proof is simpler
in this case. Indeed, it follows from Lemma~\ref{multcaplemma}, that
any direction $[s]\in\mathbb{R}P^1$ is maximizing for $d\mu$ so that
we can use the space
$$E=\left\{X_s \,\bigl|\bigr. s\in\mathbb{R}^2\right\}$$
of test functions in the variational
characterization~(\ref{varicharact}) of $\mu_2.$
Inspecting the proof of Proposition
\ref{mainestimate} we notice that the factor $2$ disappears in
\eqref{doubling} and hence in \eqref{fff} as well. Therefore, in
this case we get using \eqref{varicharact} that $\mu_2(\Omega)\le
\mu_1(\mathbb{D})$.
This completes the proof of the theorem.
\end{proof}
\begin{remark}
  When $d\mu$ is multiple, we get a stronger estimate
  $$\mu_2(\Omega)\leq\mu_1(\mathbb{D}).$$
  To illustrate this case, consider $\Omega=\mathbb{D}$. Then
  indeed $\mu_2(\mathbb{D})=\mu_1(\mathbb{D})$.
\end{remark}
\section{Proofs of auxiliary lemmas}
\subsection{Uniqueness of the renormalizing point}
The following lemma is used in the proof Lemma \ref{uniqueness}.
\begin{lemma}
\label{aux1}
 Let $r \in (0,1)$ and $s=1$. Then $X_s(d_r(z))>X_s(z)$
for all $z \in \mathbb{D}$.
\end{lemma}
\begin{proof}
\noindent We have $X_s(z)=f(|z|)\cos \theta_1$ and
$X_s(d_r(z))=f(|d_r(z)|)\cos \theta_2$, where $\theta_1=\arg z$ and
$\theta_2=\arg d_r(z)$. We need to show
\begin{equation}
\label{mainapp}
 f(|d_r(z)|)\cos \theta_2 > f(|z|)\cos\theta_1
\end{equation}
for all $z \in \mathbb{D}$. Note that the function $f$ is monotone
increasing, positive on the interval $(0,1]$, and $f(0)=0$. Set
$z=a+ib$. It is easy to check that for $|z|=0$ the inequality in
question is satisfied and therefore in the sequel we assume that
$a^2+b^2>0$.

Let us compare $|z|$ and $|d_r(z)|$. We note that $|z|=|\bar z|$.
Since $$|d_r(z)|=\frac{|z+r|}{|rz+1|},$$ we need to compare $|z+r|$
and $|r|z|^2+\bar z|$. This boils down to comparing $(a+r)^2+b^2$
and $((r(a^2+b^2)+a)^2+b^2$, or, equivalently, $(a+r)^2$ and
$((r(a^2+b^2)+a)^2$. Note that $a^2+b^2<1$ since $z \in \mathbb{D}$.
We have three cases:

\smallskip

\noindent (i) $a \ge 0$. Then $|d_r(z)|>|z|$.

\noindent (ii) $a<0$ and $a+r\le 0$. Then $|d_r(z)| < |z|$.

\noindent (iii) $a<0$ and $a+r>0$.

\smallskip

Let us now study the arguments $\theta_1$ and $\theta_2$.

We have:
$$
d_r(z)=\frac{z+r}{rz+1}=\frac{(a+r)+ib}{(ar+1)+ibr}=\frac{(a+r)(ar+1)+b^2r+ib(1-r^2)}{(ar+1)^2+b^2r^2}$$
Taking into account that $ar+1>0$, we obtain from this formula that
in case (iii) $\cos \theta_2>0$. On the other hand, $\cos
\theta_1<0$ in this case, and therefore the inequality
\eqref{mainapp} is satisfied since $f>0$.

Consider now case (i). Using the formula above we get that
$$\tan \theta_2=\frac{b(1-r^2)}{(a+r)(ar+1)+b^2r}.$$ If $a=0$ then \eqref{mainapp} is true since $\cos \theta_1=0$ and
one may easily check that $\cos \theta_2 >0$. So let us assume that
$a\ne 0$. Then $\tan\theta_1 =b/a$. Note that the tangent is a
monotone increasing function. If $b=0$ then $\theta_1=\theta_2=0$
and \eqref{mainapp} is satisfied since $|d_r(z)|>|z|$. If $b\ne 0$,
dividing by $b$ and taking into account that $a>0$, $r>0$ we easily
get:
$$
\frac{1}{a}>\frac{1-r^2}{(a+r)(ar+1)+b^2r}.
$$
Therefore, if $b>0$ we get that $\tan\theta_1 > \tan\theta_2$
implying $0<\theta_2<\theta_1<\pi/2$,  and if $b<0$ we get that
$\tan\theta_1 <\tan\theta_2$ implying that $3\pi/2<\theta_1
<\theta_2<2\pi$. At the same time, in the first case the cosine is
monotonely decreasing, and in the second case the cosine is
monotonely increasing. Therefore, for any $b\ne 0$ we get $0<\cos
\theta_1 <\cos \theta_2$, which implies \eqref{mainapp}.

Finally, consider the case (ii). If $(a+r)(ar+1)+b^2r \ge 0$ then we
immediately get \eqref{mainapp} since in this case $\cos \theta_2
\ge 0$ and $\cos \theta_1 <0$. So let us assume  $(a+r)(ar+1)+b^2r <
0$. If $b=0$ then $\theta_1=\theta_2=\pi$, hence $\cos
\theta_1=\cos\theta_2=-1$ and \eqref{mainapp} is satisfied because
$|d_r(z)| < |z|$. If $b\ne 0$, as in case (ii) we compare $\tan
\theta_1$ and $\tan \theta_2$. We claim that again
$$
\frac{1}{a}>\frac{1-r^2}{(a+r)(ar+1)+b^2r}.
$$
Since by our hypothesis the denominators in both cases are negative,
it is equivalent to $a-ar^2 < a^2r+ar^2+a+r + b^2r$. After obvious
transformations we see that this reduces to
$a^2+2ar+1+b^2=(a+r)^2+(1-r^2)+b^2
>0$ which is true.

Therefore, taking into account that tangent is monotone increasing,
we get that if $b>0$ then $\pi/2<\theta_2<\theta_1<\pi$, and if
$b<0$ then $\pi<\theta_1<\theta_2<3\pi/2$. This implies that in
either case $\cos \theta_1<\cos\theta_2<0$. Together with the
inequality $|d_r(z)| < |z|$ this gives \eqref{mainapp} in case (ii).
This completes the proof of the lemma.
\end{proof}

\subsection{Proof of  Lemma \ref{flipfloplemma}}
\label{section:flip} Let $\mathcal{M}$ be the space of signed finite
measures on $\mathbb{D}$ endowed with the norm \eqref{norm}. Recall
that the map
$\Gamma:\mathcal{M}\rightarrow\mathbb{D}$ is defined by
$\Gamma(d\nu)=\xi$ in such a way that
$d_{\xi}:\mathbb{D}\rightarrow\mathbb{D}$ renormalizes $d\nu$.
It is continuous by Corollary~\ref{cont}.
The key idea of the proof of the ``flip-flop'' lemma is to replace
the folded measure $d\mu_a$ by $$d\hat{\mu}_a:=(\tau_a)_*d\mu.$$ It is
clear that
\begin{equation}
\label{appr}
||d\mu_a-d\hat{\mu}_a||\rightarrow 0
\end{equation}
in the norm \eqref{norm} as $a$ degenerates to a point $p \in
\partial\mathbb{D}$. At the same time, the next lemma shows that
the ``flip-flop''property is true for {\it each} cap when the
rearranged measure $d\nu_a$ is replaced by
$(d\zeta_a)_*d\hat{\mu}_a$, where $\zeta_a=\Gamma(d\hat{\mu}_a)$.
\begin{lemma}
  Let $a=a_{r,p}$ be a hyperbolic cap.
  Then
  $$(d\zeta_a)_*d\hat{\mu}_a=(d_{\zeta_a})_*(\tau_a)_*d\mu=R_p^*d\mu.$$
\end{lemma}
\begin{proof}Let us show that
    $\zeta_a=-\frac{2r}{r^2+1}p$.
    Recall that $\tau_a(z)=d_{rp}\circ R_p\circ d_{-rp}.$
    A simple explicit computation then leads to
    \begin{align*}
      d_{\zeta_a}\circ\tau_a&=R_p.
    \end{align*}
    This implies
    \begin{align*}
      \int_{\mathbb{D}}X_s\circ d_{\zeta_a}\,d\hat{\mu}_a&=
      \int_{\mathbb{D}}X_s\circ d_{\zeta_a}\circ\tau_a\,d\mu\\
      &=\int_{\mathbb{D}}X_s\circ R_p\,d\mu
      =\int_{\mathbb{D}}X_{R_ps}\,d\mu=0\\
    \end{align*}
    which proves the claim.
  \end{proof}

Let $\eta_a:=\Gamma((d_{\zeta_a})_*d\mu_a)$ be the renormalizing
vector for the measure $(d_{\zeta_a})_*d\mu_a$.
\begin{lemma}
\label{lemma:eta}
  As the cap $a$ degenerates to a point $p\in\partial\mathbb{D}$,
  $\eta_a\rightarrow 0.$
\end{lemma}
\begin{proof}
 Since $d_{\zeta_a}$
is a diffeomorphism,
$(d_{\zeta_a})_*:\mathcal{M}\rightarrow\mathcal{M}$ is
  an isometry so that
  \begin{align*}
    (d_{\zeta_a})_*d\mu_a&=(d_{\zeta_a})_*(d\mu_a-d\hat{\mu}_a)+(d_{\zeta_a})_*d\hat{\mu}_a\\
    &=\underbrace{(d_{\zeta_a})_*(d\mu_a-d\hat{\mu}_a)}_{\rightarrow 0}+(\underbrace{d_{\zeta_a}\circ\tau_a}_{R_p})_*d\mu
    \rightarrow (R_p)_*d\mu.
  \end{align*}
   Here we have used \eqref{appr}. Continuity of $\Gamma$ leads to
  $$0=\Gamma((R_p)_*d\mu)=\lim_{a\rightarrow p}\Gamma((d_{\zeta_a})_*d\mu_a)=\lim_{a\rightarrow p}\eta_a.$$
   Note that the first equality follows from \eqref{conventionI} and
   the identity $X_s\circ R_p=X_{R_p s}$ that we used earlier.
\end{proof}
Set
\begin{gather}
\label{etazeta}
  q(a)=\frac{\overline{\zeta_a}\eta_a+1}{\zeta_a\overline{\eta_a}+1},\ \ \
\xi(a)=d_{\zeta_a}(\eta_a)=\left(\frac{\eta_a+\zeta_a}{\overline{\zeta_a}\eta_a+1}\right).
\end{gather}
A direct computation (cf. the proof of Lemma \ref{uniqueness}) leads
to
$$\tilde T_a(z):=d_{\eta_a}\circ d_{\zeta_a}=q(a)d_{\xi(a)}(z).$$
It follows from its definition that $\tilde T_a$ renormalizes
$d\mu_a$. Hence, $\Gamma(d\mu_a)=\xi(a)$ and $d_{\xi(a)}=T_a$, where
the transformation $T_a$ was defined in section
\ref{section:folding}. We have
\begin{align*}
  {T_a}_*d\mu_a&=(\frac{1}{q(a)}d_{\eta_a})_*(d_{\zeta_a})_*d\mu_a\\
  &=(\frac{1}{q(a)}d_{\eta_a})_*(d_{\zeta_a})_*\left(d\hat{\mu}_a+(d\mu_a-d\hat{\mu}_a)\right).
\end{align*}
Now, it follows from Lemma \ref{lemma:eta} that $\lim_{a\rightarrow
p}q(a)=1$ and $\lim_{a\rightarrow p}d_{\eta_a}=\mbox{id}$, because
$\eta_a\rightarrow 0$. Therefore, taking into account \eqref{appr}
we get
$$\lim_{a\rightarrow p}{T_a}_*d\mu_a=\lim_{a\rightarrow p}(d_{\xi_a})_*d\hat{\mu}_a=R_p^*d\mu.$$

To complete the proof of Lemma \ref{flipfloplemma} it remains to
show that as the cap $a$ degenerates to $p$,
$||{T_a}_*d\mu_a-d\nu_a||\rightarrow 0$. By definition
$d\nu_a=\psi_a^* d\mu$, where $\psi_a=(T'_a\circ\phi_b \circ
T_a)^{-1}$ (see section \ref{section:folding}).  Let us show that
$b=T_a(a) \to \mathbb{D}$ as $a \to p$. Indeed,
$$T_a=d_{\xi(a)}=
d_{\zeta_a}\circ (d_{-\zeta_a}\circ d_{\xi(a)})=R_p\circ \tau_a\circ
(d_{-\zeta_a}\circ d_{\xi(a)}).$$ Since $\eta_a \to 0$ when $a\to
p$, it follows from \eqref{etazeta} that the composition
$d_{-\zeta_a}\circ d_{\xi(a)}$ tends to identity. Therefore, the cap
$T_a(a)$ gets closer to $\mathbb{D}\setminus R_p(a)$ when $a$ goes
to $p$ and thus $\lim_{a\to p} T_a(a)=\mathbb{D}$. This implies
$\lim_{a\to p} \phi_{T_a(a)} =\mbox{id}$ and $\lim_{a\to p}
T_a'=\mbox{id}$, and hence
$\lim_{a\to p}\, ||{T_a}_*d\mu_a-d\nu_a||=0$. \qed

\section{Proof of Theorem \ref{sphere}}
\subsection{Hersch's lemma and uniqueness of the renormalizing
map}
The proof of Theorem \ref{sphere} is quite similar to the proof of
Theorem \ref{main}. We use the following notation
\begin{align*}
  \mathbb{B}^{n+1}&=\{x\in\mathbb{R}^{n+1}, |x|<1\}\\
  \mathbb{S}^n&=\partial\mathbb{B}^{n+1}.
\end{align*}
The standard round metric on $\mathbb{S}^n$ is $g_0$. Given a
conformally round metric $g \in [g_0]$ we write $dg$ for its induced
measure. Given $s\in\mathbb{R}^{n+1}$, define
$X_s:\mathbb{S}^n\rightarrow\mathbb{R}$ by
$$X_s(x)=(x,s).$$
Similarly to~(\ref{conventionI}) and~(\ref{conventionII}), we
assume that for each $s\in\mathbb{S}^n$:
\begin{gather}
  \int_{\mathbb{S}^n}X_s\, dg=0.\label{sphericalconventionI}\\
  \int_{\mathbb{S}^n}X_{e_1}^2\,dg\geq\int_{\mathbb{S}^n}X_{s}^2\,dg.\label{sphericalconventionII}
\end{gather}
Given $p\in\mathbb{S}^n$, $R_p:\mathbb{R}^{n+1}\rightarrow\mathbb{R}^{n+1}$
is the reflection with respect to the hyperplane going through
$0$ and orthogonal to the segment joining $0$ and $p$, that is
$$R_p(x)=x-2(p,x)p.$$
Given $\xi\in\mathbb{B}^{n+1}.$
define
$d_\xi:\overline{\mathbb{B}}^{n+1}
\to \overline{\mathbb{B}}^{n+1}$ by
\begin{equation}
\label{da} d_\xi(x)=\frac{(1-|\xi|^2)x+(1+2(\xi,x)
+|x|^2)\xi}{1+2(\xi,x)+|\xi|^2|x|^2}.
\end{equation}
Note that $d_\xi(0)=\xi$ and $d_\xi \circ d_{-\xi}=\mbox{id}$. The
map $d_\xi$ is  a conformal (M\"obius)  transformation of
$\mathbb{S}^n$ \cite[p. 142]{SY}. Indeed,  one can check that for
$\xi\neq 0$,
$$d_\xi=\gamma_\xi\circ R_{\frac{\xi}{|\xi|}}$$
where $\gamma_\xi$ is the spherical inversion with center
$\frac{\xi}{|\xi|^2}$ and radius $\frac{1-|\xi|^2}{|\xi|^2}$.
Note
that for $n=1$, the map $d_\xi$ coincides with the one introduced in
Lemma~\ref{renormalization}, where complex notation was used for
convenience.

Similarly to the disk case, the transformation $d_\xi$ is said to
{\it renormalize} a measure $d\nu$ on the sphere $\mathbb{S}^n$ if for each
$s\in\mathbb{R}^{n+1}$,
\begin{gather}
  \label{cmass}
  \int_{\mathbb{S}^n}X_s\circ d_\xi\,d\nu=0.
\end{gather}
This condition is clearly equivalent to
\begin{equation*}
   \int_{\mathbb{S}^n} x_i \circ d_\xi\,d\nu =0, \quad
  i=1,2,\dots, n+1,
\end{equation*}
which means that the center of mass of the measure $(d_\xi)_*d\nu$
on $\mathbb{S}^n$ is at the origin. The following result is a
combination of Hersch's lemma \cite{Hersch} and a uniqueness
result announced in \cite{Nad}.
\begin{proposition}\label{renormsphere}
  For any finite measure $d\nu$ on $\mathbb{S}^n$, there exists a
  unique point $\xi\in\mathbb{B}^{n+1}$ such that $d_\xi$ renormalizes
  $d\nu$. Moreover, the dependence of the point $\xi\in\mathbb{B}^{n+1}$
  on the measure $d\nu$ is continuous.
\end{proposition}
\begin{proof}
  The existence of $\xi$ is precisely Hersch's lemma (see \cite{Hersch}, \cite[p. 144]{SY},
  \cite[p. 274]{LY}).

  Let us prove uniqueness.
  First, let us show that if $d\nu$ is a renormalized measure then
  $\xi=0$. It follows from \eqref{da} by a straightforward computation that if
  $\mathbb{B}^{n+1}\ni \xi\ne 0$ then $X_\xi(x) < X_\xi(d_\xi(x))$
  for any
  $x\in \mathbb{S}^n$. Assume that $d_\xi$ renormalizes $d\nu$ for
  some $\xi\ne 0$. Then
  $$
  0=\int_{\mathbb{S}^n} X_\xi\,d\nu < \int_{\mathbb{S}^n} X_\xi \circ
  d_\xi\,d\nu =0,
  $$
  and we get a contradiction.

  Now, let $d\nu$ be an arbitrary finite measure and suppose that it is
  renormalized by $d_\xi$ and $d_\eta$. Writing
  $d_\eta=d_\eta \circ d_{-\xi} \circ d_\xi$ we get
\begin{equation}
\label{dc} \int_{\mathbb{S}^n} X_s \circ d_\eta \circ d_{-\xi}\,d\tilde
\sigma=0
\end{equation}
where the measure $d\tilde \sigma=(d_\xi)_* d\sigma$ is renormalized.
At the same time, it easy to check that  $d_\eta \circ d_{-\xi} = R \circ
d_{d_{-\xi}(\eta)}$, where $R$ is an orthogonal transformation. Indeed,
since $-d_{-\xi}(\eta)=d_\xi(-\eta)$ we have
$$d_\eta \circ d_{-\xi} \circ d_{d_{\xi}(-\eta)}(0)=d_\eta(-\eta)=0,$$
and it is well known that any M\"obius transformation of the unit
ball preserving the origin is orthogonal \cite[Theorem
3.4.1]{Be}. Since $R$ preserves the center of mass at zero, it
follows from \eqref{dc} that $d_{d_{-\xi}(\eta)}$ renormalizes the
measure $d\tilde \sigma$, which is already renormalized. Therefore, as
we have shown above, $d_{-\xi}(\eta)=0$ and hence $\xi=\eta$.

Similarly to Corollary \ref{cont}, uniqueness of the renormalizing
point implies that its dependence on the measure is continuous.
\end{proof}

\subsection{Spherical caps,  folding and rearrangement}
The set $\mathcal{C}$ of all spherical caps is parametrized as
follows: given $p\in\mathbb{S}^n$ let
$$a_{0,p}=\left\{x\in\mathbb{S}^n\ : (x,p)>0\right\}$$
be the half-sphere centered at $p$. Given $-1<r<1$, let
$$a_{r,p}=d_{rp}(a_{0,p}).$$
To every spherical cap $a \in \mathcal{C}$ we associate a {\it folded}
measure:
$$d\mu_a=\begin{cases}
    dg+\tau_a^*dg & \mbox{ on } a,\\
    0 & \mbox{ on } a^*,
  \end{cases}$$
where $a^*=\mathbb{S}^n\setminus\overline{a} \in
\mathcal{C}$ is the cap adjacent to $a$, and $\tau_a$ is
the conformal reflection with respect to the boundary circle of
$a$. That is, for $a=a_{r,p}$
$$\tau_a=d_{rp}\circ R_p\circ d_{-rp}.$$
Let $\xi(a)\in\mathbb{B}^{n+1}$ be the unique point such that
$d_{\xi(a)}$ renormalizes $d\mu_a$. We obtain a
\emph{rearranged folded measure}
\begin{gather}\label{rearangedsphericalmeasure}
  d\nu_a = (d_{\xi(a)})_* d\mu_a.
\end{gather}

\subsection{Maximizing directions}
Given a finite measure  $d\nu$ on $\mathbb{S}^n$, define
$$V(s)=\int_{\mathbb{S}^n} X_s^2 d\nu.$$
Let $\mathbb{R}P^n$ be the projective space and let $[s] \in
\mathbb{R}P^n$ be the point corresponding to $\pm s \in
\mathbb{S}^n$.
We say that $[s] \in \mathbb{R}P^n$ is a {\it
maximizing direction} for  $d\nu$ if $V(s)
\ge V(t)$ for all $t \in \mathbb{S}^n$. We say that the
spherical cap is {\it simple} if the maximizing direction is unique.
Otherwise, similarly to Lemma \ref{multcaplemma}, there exists a
two-dimensional subspace $W\subset\mathbb{R}^{n+1}$ such that any
$s\in W\cap\mathbb{S}^n$ is a maximizing direction for $d\nu$. In
particular for each $s,t\in W$, $V(s)=V(t).$
In this case the measure $d\nu$ is called {\it multiple}.
\begin{proposition}\label{sphericalmultexist}
Let $g\in[g_0]$ be a conformally round metric on a sphere
$\mathbb{S}^n$ of odd dimension. If the measure $dg$ is simple then
there exists a spherical cap such that the rearranged folded measure
$d\nu_a$ is multiple.
\end{proposition}
The proof of Proposition~\ref{sphericalmultexist} is similar to the
proof of Proposition~\ref{multexist}.
We assume the measures $dg$ as well as each $d\nu_a$ to be
simple. Given a cap $a\subset\mathbb{S}^n$ let
$[s(a)]\in\mathbb{R}P^1$ be the unique maximizing direction for
$d\nu_a$. The map $a\mapsto [s(a)]$ is continuous.
The following spherical version of the ``flip-flop'' property is
proved exactly as Lemma~\ref{flipfloplemma}.
\begin{lemma}\label{sphericalflipfloplemma}
  If a sequence of spherical caps $a\in\mathcal{C}$  degenerates to a
  point $p\in\mathbb{S}^n$, the limiting rearranged measure is a
  ``flip-flop'' of the original measure $dg$:
  \begin{gather}
    \lim_{a\rightarrow p}d\nu_a=R_p^*dg. 
  \end{gather}
\end{lemma}
Similarly to Lemma~\ref{degeneratecaps} we study the maximizing
directions for degenerating caps.
\begin{lemma}\label{sphericaldegeneratecaps}
  Suppose the measures $dg$ as well as each $d\nu_a$ are simple.
  Then
  \begin{gather}
    \lim_{a\rightarrow\mathbb{S}^n} [s(a)]=[e_1]\label{sphericaleqndegenerate1}\\
    \lim_{a\rightarrow p} [s(a)]=
    [R_pe_1]\label{sphericaleqndegenerate2}.
  \end{gather}
\end{lemma}

\smallskip

\begin{proof}[Proof of Proposition~\ref{sphericalmultexist}]
By convention~(\ref{sphericalconventionII}), $[e_1]$ is the unique
maximizing direction for $dg$.
Recall that the space of caps has been identified with
$(-1,1)\times\mathbb{S}^n.$
The continuous map
$$h:[-1,1]\times\mathbb{S}^n\rightarrow\mathbb{R}P^n$$
is defined by
\begin{gather*}
  h(r,p)=
  \begin{cases}
    [e_1] & \mbox{ for } r=-1,\\
    [s(a_{r,p})] & \mbox{ for } -1<r<1,\\
    [R_pe_1] & \mbox{ for } r=1.
  \end{cases}
\end{gather*}
That is, $h$ is an homotopy between a constant map and the map
$$\phi:\mathbb{S}^n\rightarrow\mathbb{R}P^n$$
 defined by
$\phi(p)=[R_pe_1].$
We will show that this is impossible when $n$ is odd by computing its
degree.
The map $\phi$ lifts to the map $\psi:\mathbb{S}^n\rightarrow\mathbb{S}^n$
defined by
\begin{gather}\label{formulapsi}
  \psi(p)=-R_pe_1=2(e_1,p)p-e_1.
\end{gather}
The two solutions of $\psi(p)=e_1$ are $e_1$ and $-e_1$. It is
easy to check that since the dimension $n$ is odd, both
differentials
\begin{gather*}
  D_{e_1}\psi:T_{e_1}\mathbb{S}^n\rightarrow T_{e_1}\mathbb{S}^n\\
  D_{-e_1}\psi:T_{-e_1}\mathbb{S}^n\rightarrow T_{e_1}\mathbb{S}^n
\end{gather*}
preserve the orientation. This implies $\mbox{deg}(\psi)= 2.$
Moreover, the quotient map
$\pi:\mathbb{S}^n\rightarrow\mathbb{S}^n$ has degree 2 for $n$
odd. It follows that
\begin{align*}
  \mbox{deg}(\phi)&=\mbox{deg}(\pi\circ\psi)\\
  &=\mbox{deg}(\pi) \mbox{deg}(\psi)=4.
\end{align*}
Since the degree of a map is invariant under homotopy, this is a
contradiction.
\end{proof}
\begin{remark}
\label{whyodd} In even dimensions one of the differentials $D_{\pm
e_1}$ preserves the orientation and the other reverses it.
Therefore, $\mbox{deg}(\phi)=0$ and the proof of Proposition
\ref{sphericalmultexist} does not work in this case. In dimension
two the existence of a multiple cap was proved in \cite{Nad} using a
more sophisticated topological argument.
\end{remark}

\subsection{Test functions and the modified Rayleigh quotient}
\label{modR} Let $g_0$ be the standard round metric on the sphere
$\mathbb{S}^n$, so that
\begin{equation}
\label{sphere:volume} \omega_n:=\int_{\mathbb{S}^n}dg_0=
\frac{2\pi^{\frac{n+1}{2}}}{\Gamma(\frac{n+1}{2})}.
\end{equation}

Let $g\in[g_0]$ be a conformally round Riemannian metric of volume
one, that is $\int_{\mathbb{S}^n}\,dg=1.$ The Rayleigh quotient of a
non-zero function $u\in H^1(\mathbb{S}^n)$ is
$$R(u)=
\frac{\int_{\mathbb{S}^n}|\nabla_g u|_g^2\,dg}
{\int_{\mathbb{S}^n}u^2\,dg}.$$
We use the following variational characterization of $\lambda_2(g)$:
\begin{gather}\label{varicharact1}
  \lambda_2(g)=\inf_{E}
  \sup_{0\neq u\in E}R(u)
\end{gather}
where $E$ varies among all two-dimensional subspaces of the Sobolev
space $H^1({\mathbb{S}^n})$ that are orthogonal to constants, in the
sense that for each $f\in E$, $\int_{\mathbb{S}^n}f\,dg=0.$
Following~\cite{FN}, we use a \emph{modified Rayleigh quotient}:
$$R'(u)=\frac{\left(\int_{\mathbb{S}^n}|\nabla_g u|_g^n\,dg\right)^{2/n}}
{\int_{\mathbb{S}^n}u^2\,dg}.$$ It follows from Holder inequality
that $R(u)\leq R'(u)$ for each $0\neq u\in H^1(\mathbb{S}^n)$. It is
easy to check that $\int_{\mathbb{S}^n}|\nabla_g u|_g^n\,dg$ is
conformally invariant for each dimension $n$ so that we can rewrite
the modified Rayleigh quotient as follows:
$$R'(u)=\frac{\left(\int_{\mathbb{S}^n}|\nabla u|^n\,dg_0\right)^{2/n}}
{\int_{\mathbb{S}^n}u^2\,dg}$$
where the gradient and it's norm are with respect to the round metric
$g_0$.

\smallskip

Assume that $dg$ is simple 
and let $a\subset\mathbb{S}^n$ be a spherical cap such that $d\nu_a$
is multiple.  Let $W\subset\mathbb{R}^{n+1}$ be the corresponding
two dimensional subspace of maximizing directions. Given a function
$u:a\rightarrow\mathbb{R}$, the \emph{lift} of $u$,
$\tilde{u}:\mathbb{S}^n\rightarrow\mathbb{R}$ is defined exactly as
in Definition~\ref{defLift}.
\begin{proposition}
\label{modified}
  Given $s\in W\subset \mathbb{R}^{n+1}$, the function
  $u_a^s=X_s\circ d_{\xi(a)}:a\rightarrow\mathbb{R}$
  is such that
  $$R'(\tilde{u}_a^s)<(n+1)\left(4
    \frac{\pi^{\frac{n+1}{2}}\Gamma(n)}{\Gamma(\frac{n}{2})\Gamma(n+\frac{1}{2})}
  \right)^{2/n}.$$
\end{proposition}
\begin{proof}
  The conformal invariance of the numerator in $R'(u)$ implies
  \begin{multline}
  \label{gradu}
    \left(\int_{\mathbb{S}^n}|\nabla_g \tilde{u}_a^s|_g^n\,dg\right)^{2/n}=
\left(\int_{a}|\nabla_g u_a^s|_g^n\,dg\right)^{2/n}+\\
\left(\int_{a^*}|\nabla_g (u_a^s \circ
\tau_a)|_g^n\,dg\right)^{2/n}=
    \left(2\int_{a}|\nabla_g u_a^s|_g^n\,dg\right)^{2/n} \\
    =\left(2\int_{d_{\xi(a)}(a)}|\nabla_g X_s|_g^n\,dg\right)^{2/n}
    <
\left(2\int_{\mathbb{S}^n}|\nabla_{g_0}
      X_s|_{g_0}^n\,dg_0\right)^{2/n}
  \end{multline}
  Here the second equality follows from conformal invariance. To obtain the inequality at the end
  we again use the conformal invariance as well as the fact that
  $d_{\xi(a)}(a) \subsetneq
  \mathbb{S}^n$.
  To estimate the denominator in the modified Rayleigh quotient we
  first note that for any $x=(x_1,\dots x_{n+1}) \in \mathbb{S}^n$,
  $$\sum_{j=1}^{n+1} \tilde{u}_a^{e_j}(x)^2=\sum_{j=1}^{n+1} x_j^2=1.$$
  Therefore, given that $\int_{\mathbb{S}^n} dg=1$ we obtain:
  \begin{align*}
    \sum_{j=1}^{n+1}\int_{\mathbb{S}^n}
    (\tilde{u}_a^{e_j})^2 \,dg=1
  \end{align*}
  Now, since $W$ is a subspace of maximizing directions for the measure $d\nu_a$ defined by
  \eqref{rearangedsphericalmeasure}, for each $s\in W$ we have
\begin{equation}
\label{denomin}
   \int_{\mathbb{S}^n}(\tilde{u}_a^s)^2 \,dg\geq\frac{1}{n+1}.
  \end{equation}
Set
  $$K_n:=
  \int_{\mathbb{S}^n}|\nabla_{g_0}X_s|_{g_0}^n\,dg_0.$$
Combining \eqref{gradu} and \eqref{denomin} we get
\begin{equation}
\label{rprime}
  R'(\tilde{u}_a^s)\le (n+1)\left(2K_n\right)^{2/n}.
\end{equation}
Proposition \ref{modified} then follows from the lemma below.
\begin{lemma}
\label{kn} The constant $K_n$ is given by
$$K_n=\frac{2\pi^{\frac{n+1}{2}}\Gamma(n)}{\Gamma(\frac{n}{2})\Gamma(n+\frac{1}{2})}.$$
\end{lemma}
\noindent {\it Proof.} Recall that $g_0$ is the standard round
metric on the unit sphere $\mathbb{S}^n$. If we consider
$X_s(x)=(x,s)$ as a function on $\mathbb{R}^{n+1}$ then its gradient
is just the constant vector $s$:
$$\mbox{grad}_{\mathbb{R}^{n+1}}X_s=s.$$
This means that for any point $p \in \mathbb{S}^n$ the gradient of
the function $X_s:\mathbb{S}^n\rightarrow\mathbb{R}$ at $p$  is the
projection of $s$ on the tangent space $T_p\mathbb{S}^n$:
$$\nabla X_s(p)=s-(s,p)p.$$
Therefore, taking into account that $|s|=|p|=1$, we get
$$
  |\nabla X_s(p)|^n
  =(|s-(s,p)p|^2)^{n/2}
  =(1-(s,p)^2)^{n/2},
$$
and hence
\begin{align*}
  K_n=\int_{\mathbb{S}^n}(1-(s,p)^2)^{n/2}dg_0.
\end{align*}
Let $\theta$ be the angle between the vectors $p$ and $s$. 
Making a change of variables we obtain
$$
  K_n=\omega_{n-1} \int_{0}^{\pi}(1-\cos^2 \theta)^{n/2} (\sin \theta)^{n-1}\,d\theta
  =\omega_{n-1}\int_{0}^{\pi}\sin^{2n-1} \theta \,d\theta,
  $$
where $\omega_{n-1}$ is the volume of the standard round sphere
$\mathbb{S}^{n-1}$ given by \eqref{sphere:volume}.

The calculation of a table integral \cite[3.621(4)]{GR}
$$\int_{0}^{\pi}\sin^{2n-1} \theta \,d\theta=\frac{\sqrt{\pi}\,\,\Gamma(n)}{\Gamma(n+\frac{1}{2})}$$
completes the proofs of Lemma \ref{kn} and Proposition
\ref{modified}.
\end{proof}
\begin{remark}
\label{whynotsharp} It follows from H\"older inequality that
$R(u)=R'(u)$ if and only if $u$ is a constant function. Since
$\nabla_{g_0}X_s \neq \mbox{const}$ we get a {\it strict} inequality
$R'(\tilde{u}_a^s) > R(\tilde{u}_a^s)$. This is why the estimate
\eqref{sphere:bound} is not sharp. In the context of the first
eigenvalue, a similar difficulty was encountered in \cite[Lemma
4.15]{Berger}) and overcame in \cite{EI}. To apply the approach of
\cite{EI} we need to have a spherical cap of multiplicity $n+1$;
existence of a cap of multiplicity two proved in Proposition
\ref{sphericalmultexist} is not enough for this purpose.
\end{remark}


\begin{proof}[Proof of Theorem~\ref{sphere}] If the measure $dg$ is simple, then \eqref{sphere:bound} follows
from Proposition \ref{modified} and the variational principle
\eqref{varicharact1}. If $dg$ is multiple, then, as in the proof of
Theorem \ref{maintheorem} at the end of section
\ref{subsection:testfunctions}, one can work directly with this
measure without any folding and rearrangement. Inspecting the proof
of Proposition \ref{modified} we notice that the factor $2^{2/n}$
disappears in \eqref{gradu} and hence also in \eqref{rprime}.
Therefore, in this case we get an even better bound than
\eqref{sphere:bound}. This completes the proof of
Theorem~\ref{sphere}.
\end{proof}

\end{document}